\documentclass[12pt]{huber_article}

\usepackage{amsmath,amssymb,amsthm}
\usepackage{booktabs}
\usepackage{comment}

\newcommand{\upperbound}
   {9.5\epsilon^{-1}C}
\newcommand{\lowerbound}
  {(3/16)e^{-2}\ln(7) C \epsilon^{-1} (1 - \epsilon^{-1/2})^2(1 - (1 - \epsilon)/C)}
\newcommand{\kasfunctionofepsilonandC}
  {2.3(\gamma\epsilon)^{-1}}

\newcommand{\mean}{\mathbb{E}}

\newcommand{\prob}{\mathbb{P}}

\newcommand{\sd}{\operatorname{SD}}
\newcommand{\unif}{\textsf{Unif}}

\newcommand{\bern}{\textsf{Bern}}
\newcommand{\geo}{\textsf{Geo}}
\newcommand{\ind}{\textbf{1}}
\newcommand{\defeq}{\mathrel{\mathop{=}\limits^{\textrm{def}}}}

\begin{document}

\newtheorem{theorem}{Theorem}{
\newtheorem{lemma}{Lemma}
\theoremstyle{definition}
\newtheorem{definition}{Definition}

\title{Nearly optimal Bernoulli factories for \\ linear functions}

\author{Mark Huber \\ {\tt mhuber@cmc.edu}}

\maketitle

\begin{abstract}
Suppose that $X_1,X_2,\ldots$ are 
independent identically distributed Bernoulli random variables
with mean $p$.  
A Bernoulli factory for a function $f$ takes as input $X_1,X_2,\ldots$ 
and outputs
a random variable that is Bernoulli with mean $f(p).$  A fast algorithm
is a function that only depends on the values of 
$X_1,\ldots,X_T$, where $T$ is a stopping time with small
mean.  When $f(p)$ is a real analytic function the problem reduces
to being able to draw from linear functions $Cp$ for a constant $C > 1$.
Also it is necessary that 
$Cp \leq 1 - \epsilon$ for known $\epsilon > 0$.
Previous methods for this problem required extensive modification of 
the algorithm
for every value of $C$ and $\epsilon$.  
These methods
did not have explicit bounds on $\mean[T]$ as a function of 
$C$ and $\epsilon$.  This paper presents the first Bernoulli factory
for $f(p) = Cp$ 
with bounds on $\mean[T]$ as a function of the input 
parameters.  In fact,
$\sup_{p \in [0,(1-\epsilon)/C]} \mean[T] 
  \leq \upperbound.$
In addition, this method is very simple to implement.
Furthermore, a lower bound on the average running time of any 
$Cp$ Bernoulli factory is shown.
For 
$\epsilon \leq 1/2$, $\sup_{p \in [0,(1 - \epsilon)/C]} \mean[T] 
  \geq 0.004 C \epsilon^{-1}$, 
so the new method is optimal up to 
a constant in the running time.
\end{abstract}

\vspace*{5pt}
\vspace*{5pt}

\section{Introduction}

Let $\bern(p)$ denote the Bernoulli distribution with mean $p$, so
$X \sim \bern(p)$ means $\prob(X = 1) = p$ and $\prob(X = 0) = 1 - p$.
A Bernoulli random variable will also be referred to as a coin flip, with
heads equating to the value 1 and tails to the value 0.
Let $X_1,X_2,\ldots$ be an infinite sequence of independent, identically
distributed (iid) $\bern(p)$ random variables.

Following~\cite{keaneo1994}, 
a Bernoulli factory is an algorithm that takes independent realizations
of $X \sim \bern(p)$ and auxiliary variables with known distributions,
and simulates a Bernoulli variable with success probability 
$f(p)$.

Recall that $U \sim \unif([0,1])$ can be viewed as an iid sequence 
of $\bern(1/2)$ random variables simply by reading off the bits in 
the number $U$.  These bits can then be used to build an iid sequence of 
uniform random numbers in $[0,1]$.  Hence a single $U$ represents 
for notational purposes all the extra randomness utilized by any 
randomized algorithm.

\begin{definition}
Given $p^* \in (0,1]$ and a function $f:[0,p^*] \rightarrow [0,1]$, 
a {\em Bernoulli factory} is a computable 
function $\cal A$ that takes as input a number $u \in [0,1]$ together
with a sequence of values in $\{0,1\}$, 
and returns an output in $\{0,1\}$
where the following holds.  For any $p \in [0,p^*]$, 
$X_1,X_2,\ldots$ iid $\bern(p)$, and $U \sim \unif([0,1])$,
let $T$ be the infimum of times $t$ such that the value of 
${\cal A}(U,X_1,X_2,\ldots)$ only depends on the values of $X_1,\ldots,X_t$.
Then
\begin{enumerate}
\item{$T$ is a stopping time with respect to the natural
filtration and $\prob(T < \infty) = 1$.}
\item{${\cal A}(U,X_1,X_2,\ldots) \sim \bern(f(p))$.}
\end{enumerate}
Call $T$ the {\em running time} of the Bernoulli factory.
\end{definition}
Informally, ${\cal A}$ is a Bernoulli factory if it takes $T$ 
flips of a coin with $p \in [0,p^*)$ 
chance of heads together with a source of random bits, 
and returns a single coin flip that has 
probability $f(p)$ chance of heads.

The applications of Bernoulli factories go back at least as far as 
Von Neumann~\cite{vonneumann1951}, who showed how to build such a 
factory for $p^* = 1$ and $f(p) = 1/2$.  His method works as follows.
\begin{align*}
T &= \inf\{t \in \{2,4,\ldots\}:(X_{t - 1},X_{t}) \in \{(0,1),(1,0)\} \\
Y &= \ind((X_{T - 1},X_T) = (0,1)), 
\end{align*}
where $\ind(\text{expression})$ 
denotes the indicator function that evaluates to 
1 if the expression is true and 0
otherwise.  

In 1992 Asmussen raised the question of whether it was possible to 
construct a Bernoulli factory for $f(p) = Cp$, the application being
perfect sampling for certain positive recurrent regenerative 
processes~\cite{asmussengt1992}.

Keane and O'Brien~\cite{keaneo1994} showed that a Bernoulli factory
where $T$ is almost surely finite exists 
if and only if $f(p)$ is continuous over its domain and either
$f(p)$ is  
identically 0 or 1, or both $f(p)$ and $1 - f(p)$ are polynomially 
bounded away from 0 over the domain.  They gave an explicit algorithm
for building a Bernoulli factory in such a case using the fact that
continuous $f(p)$ could be well-approximated by Bernstein polynomials
(linear combinations of 
$p^k(1 - p)^{n - k}$ where $n$ and $k \leq n$ are integers.)
However, they did not prove
any bounds on $\mean[T]$ (upper or lower).

Nacu and Peres further refined the approach of 
Keane and O'Brien, showing that for $\epsilon > 0$,
and $p^* = (1 - \epsilon)/2$, the function $f(p) = 2p$ had a Bernoulli
factory where $T$ has exponentially declining tails
(Theorem 1 of~\cite{nacup2005}).  Moreover, they showed that 
in a sense $f(p) = 2p$
is the most important case, as a Bernoulli factory for $f(p) = 2p$
can be used to build a Bernoulli factory for any function that is 
real analytic in $(0,1)$ and bounded away from 1.

{\L}atuszy\'nski et. al~\cite{latuszynskikpr2011} 
gave the first practical implementation of
the Nacu and Peres approach.  Instead of dealing with sets of outcomes
of coins, they created a continuous approach using a sequence of functions
$L_n$ that lower bound (in expectation) $f(p)$, and $U_n$ that upper bound
(in expectation) $f(p)$.  $L_n$ is a reverse time supermartingale, while
$U_n$ is a reverse time submartingale, both of which monotonically converge
to $f(p)$ as the number of coins grows to infinity.  This allowed the 
algorithm to proceed without having a keep track of an (exponentially growing)
number of possible sequences of coin flips, and allowed the Nacu-Peres
algorithm to actually be implemented.

Flegal and Herbei~\cite{flegalh2012} developed a faster algorithm
by changing $f(p)$.  In the work of Nacu and Peres~\cite{nacup2005}, 
instead of restricting themselves to $Cp \leq 1 - \epsilon$, they 
used $f(p) = \min\{Cp,1-\epsilon\}$ so that the function was defined
over the entirety of $[0,1]$.  Of course this makes the derivative of
$f(p)$ discontinuous at $p = (1 - \epsilon)/C$.  
The idea of Flegal and Herbei was to
handle the $p \in [(1-\epsilon)/C,1]$ situation
differently, in such a way that
$f(p)$ became twice differentiable at $p = (1 - \epsilon)/C$.  
They then employed the {\L}atunszy\'nski et. al.~\cite{latuszynskikpr2011}
reverse time martingale method to actually perform the simulations.

Experimental results in~\cite{flegalh2012} 
indicated that this greatly improved upon the running time, although
the running time still appeared to be superlinear in $C$.

With the positive recurrent regenerative process application,
it makes sense to measure the running
time of the algorithm in terms of the number of coin flips, since 
for applications flipping a single coin involves running a Markov chain
for a large number of steps.  In other words, the coin flips dominate
the time needed to compute $\cal A$ by a large factor.

Thomas and Blanchet~\cite{thomasb2011} gave a different improvement to the 
Nacu and Peres algorithm, but did not analyze the running time.
Their numerical experiments indicate that
their algorithm also does not scale linearly with $C$.  As with the other
approaches mentioned, this was a variation of 
the original Keane and O'Brien method, which relies
on approximating the 
$f(p)$ function using Bernstein polynomials.

The novel approach presented here is very different.
Throughout the rest of this work, assume that $f(p) = Cp$, where 
$C > 1$, $p < (1 - \epsilon)/C$, and $\epsilon > 0$.
The new method operates by flipping coins or utilizing extra randomness
to change the target function $f(p)$ from step to step.
Moreover, it has a form that
allows $C$ and $\epsilon$ to change without requiring extensive calculation of 
new parameter values. This allows for the first time a direct computation
of an upper bound for 
$\mean[T]$.

\begin{theorem} 
\label{THM:upperbound}
The algorithm of Section~\ref{SEC:algorithm} 
is a Bernoulli factory with running time
\[
\sup_{p \in [0,(1 - \epsilon)/C]} \mean[T] \leq \upperbound.
\]
\end{theorem}

(Theorem~\ref{THM:preciseupperbound} in Section~\ref{SEC:upper} gives
a more precise upper bound on $\mean[T]$.)
This new approach has the correct order of running time, 
as shown by the following lower bound on the average number of flips.

\begin{theorem}
\label{THM:lowerbound}
Any Bernoulli factory has 
\[
\sup_{p \in [0,(1 - \epsilon)/C]} \mean[T] \geq \lowerbound.
\]
\end{theorem}

The remainder of the paper is organized as follows.  
Section~\ref{SEC:algorithm} presents the new method and proves 
correctness.
Section~\ref{SEC:upper} proves the precise 
upper bound on $\mean[T]$ and compares this upper bound
to simulations of the algorithm.  Then
Section~\ref{SEC:lower} proves the lower bound on $\mean[T]$.

\section{Algorithm}
\label{SEC:algorithm}

This section is organized as follows.  First the intuition of the
new method is presented.  Next the pseudocode is given based of off the
intuition, and finally the rigorous details are given.

\subsection{Intuition}

Suppose that instead of flipping one coin, first 
flip $B \sim \bern(p)$, and then another Bernoulli $Y$ whose
mean depends on $B$.  In particular, $[Y|B = 1] \sim \bern(p_1)$, and
$[Y|B = 0] = \bern(p_0)$.  Then
\begin{equation*}
\prob(Y = 1) = p p_1 + (1 - p) p_0.
\end{equation*}

Our goal is to make $\prob(Y = 1) = Cp$ where $C > 1$.  Let $p_1 = 1,$
then to make this work $p_0 = (C - 1)p/(1 - p)$ so that
\begin{equation*}
p(1) + (1 - p)(C - 1)p/(1 - p) = Cp.
\end{equation*}

At this point the problem of flipping a $Cp$ coin has been reduced to 
flipping a $(C - 1)p/(1 - p)$ coin.  Now employ a trick in Nacu and 
Peres~\cite{nacup2005} used there to reduce a Bernoulli factory problem
for real analytic $f(p)$ to a Bernoulli factory for $2p$.
The trick is to write $(C - 1)p/(1 - p)$ using the power series 
expansion.
\begin{equation*}
(C - 1)\frac{p}{1 - p} = (C - 1)p + (C - 1)p^2 + (C - 1)p^3 + \cdots.
\end{equation*}

Let $G \sim \geo(a)$ denote the geometric distribution, so 
for positive integers $g$, $\prob(G = g) = (1 - a)^{g - 1}a.$  Let
$G \sim \geo((C - 1)/C)$.  Then let $[W|G] \sim \bern((Cp)^G)$.  Then
\[
\prob(W = 1) = \sum_{i=1}^\infty \prob(G = i)(Cp)^i
 = \sum_{i=1}^\infty 
   \left(\frac{1}{C}\right)^{i-1}\left(\frac{C - 1}{C}\right)(Cp)^i
 = \sum_{i=1}^\infty (C - 1)p^i.
\]

So the problem has been reduced from generating a $(C - 1)p/(1 - p)$ 
coin to generating some positive number of $Cp$ coins.  This perhaps
seems circular, as this is the problem that we started with!  However,
something is gained here.

To see what, 
consider a concrete example.  Suppose that $\epsilon = 0.01$, so that
$Cp \leq 1 - 0.01$.  Suppose further than the problem of generating
one $Cp$ coin has been changed into the problem of generating a
$(Cp)^{100}$ coin.  Since $(Cp) \leq 1 - 0.01$, $(Cp)^{100} < e^{-1}$.
So there is at least a one half chance that the coin flip will result in
a 0.  Suppose that $R \sim \bern(1/2)$, and that $[Z|R = 1] \sim \bern(0)$,
and $[Z|R = 0] \sim \bern(r_0)$.  Then
\[
\prob(Z = 1) = (1/2)(0) + (1/2)r_0,
\]
so to make $\prob(Z = 1) = (Cp)^{100}$, $r_0 = 2(Cp)^{100} = (2^{1/100}C p)^{100}.$
In other words, in this case it is necessary to simulate 100 coins
from the Bernoulli factory $(2^{1/100} C) p$.  Note that $2^{1/100}$ is 
very close to 1, so $2^{1/100} C$ is not too far away from the original $C$.

\subsection{Pseudocode}

There are three phases that the algorithm goes through.  A single $p$-coin
flip gets us started, either moving straight to heads, or to 
a $(C - 1)p/(1-p)$-coin.  A draw of a geometric random variable then
allows us to reduce the $(C - 1)p/(1-p)$-coin to a $(Cp)^i$-coin for
some $i$.  That of course is the same as flipping a $Cp$-coin followed
by a $(Cp)^{i-1}$ coin.  Flip the $Cp$ coin in the same way, to either
get a head or get a $(Cp)^{j+1}$ coin, which combines with the remaining coins
to give a $(Cp)^{j + i}$ coin.  

Always after the $p$-coin flip and the Bernoulli, the remaining task is 
to flip a $(Cp)^i$ coin for some integer $k$.  Hopefully $i = 0$ at some
point, in which case the answer is heads since $(Cp)^0 = 1$.  
On the other hand, if $i$ gets
too large ($4.6 \epsilon^{-1}$ to be precise), flip a $(1 + \epsilon/2)^{-1}$-
coin.  If that coin comes up 0, return the answer $0$, otherwise the new
goal is to flip a $(C(1 + \epsilon/2) p)^i$-coin.
Notes that $C(1+\epsilon/2)p \leq 1 - \epsilon/2$, so our bound away from
1 has shrunk by a factor of 2.

This procedure is encapsulated in the following pseudocode.

\begin{center}
\begin{tabular}{rl}
\toprule
\multicolumn{2}{l}{{\sc Linear\_Bernoulli\_Factory}\ \  {\em Input:} 
 $C$, $\bern(p)$, $\epsilon$} \\
\midrule
1) & \hspace*{0em} $\gamma \leftarrow 1/2, \ 
     k \leftarrow \kasfunctionofepsilonandC$, 
     $i \leftarrow 1$, $\epsilon \leftarrow \min\{\epsilon,0.644\}$ \\
2) & \hspace*{0em} Repeat \\
3) & \hspace*{1em} Repeat \\
4) & \hspace*{2em} Draw $B \leftarrow \bern(p)$, $G \leftarrow \geo((C-1)/C)$ \\
5) & \hspace*{2em} $i \leftarrow i - 1 + (1 - B)G$ \\
6) & \hspace*{1em} Until $i = 0$ or $i \geq k$ \\
7) & \hspace*{1em} If $i \geq k$ \\
8) & \hspace*{2em}    Draw $R \leftarrow \bern((1 + \gamma \epsilon)^{-i})$ \\
9) & \hspace*{2em}    $C \leftarrow C(1 + \gamma \epsilon)$, 
  $\epsilon \leftarrow (1 - \gamma) \epsilon$, $k \leftarrow k/(1 - \gamma)$. \\
10) & \hspace*{0em}  Until $i = 0$ or $R = 0$. \\
11) & \hspace*{0em} Return $\ind(i = 0)$ \\
\bottomrule
\end{tabular}
\end{center}

\subsection{Details}

To prove the algorithm works, begin by building a stochastic process
$M_0,M_1,M_2,\ldots$ where
$M_t$ represents the mean of the Bernoulli to be drawn.  Begin with
$M_0 = Cp$.  

\begin{lemma}
\label{THM:martingale}
Let $M_0,M_1,\ldots$ be a martingale on $[0,1]$ with $M_0 = \theta$.
Let $T(\theta) = \inf\{t:M_t \in \{0,1\}\}$.  If
$\prob(T(\theta) < \infty) = 1$, then 
\[
M_{T(\theta)} \sim \bern(\theta).
\]
\end{lemma}

\begin{proof}
Since $M_t$ is a bounded (and hence uniformly integrable) martingale,
the Optional Stopping Theorem applies 
(see for instance~\cite[p. 269]{durrett2005}) and 
$\mean[M_{T(\theta)}] = \mean[M_0] = \theta.$  Since 
$\prob(T(\theta) < \infty) = 1$,
$M_{T(\theta)}$ is a $\{0,1\}$ random variable with mean $\theta$, so
$M_{T(\theta)} \sim \bern(\theta)$.
\end{proof}

The next lemma presents the moves the martingale will take.
\begin{lemma}
\label{LEM:markov}
Consider a stochastic process $\{M_k:k \geq 0\}$ satisfying the following.
For each $t$, 
given $M_0,\ldots,M_{t-1}$, the value of $M_t$ is chosen according to either
the transition kernel ${\bf K}_{1,C}(\cdot)$ or the transition kernel
$K_{2,\alpha}$.  The kernel $K_{1,C}$ is defined as
\begin{align*}
{\bf K}_{1,C}((Cp)^i,(Cp)^{i-1}) &= p \\
{\bf K}_{1,C}((Cp)^i,(Cp)^{i + j}) &= (1 - p)(C - 1)C^{-j-1}, & 
  \text{where } j \in \{0,1,2,\ldots\}.
\end{align*}
(Note this kernel can only be applied when $M_{t-1}$ is of the form
$(Cp)^i$ for some $i \in \{1,2,\ldots\}$.)

The kernel $K_{2,\alpha}$ is defined (for $\alpha > 1$ and $0 < x \leq 1/\alpha$
) as
\begin{align*}
{\bf K}_{2,\alpha}(x, \alpha x) &= 1 / \alpha,\\ 
{\bf K}_{2,\alpha}(x,0) &= 1 - 1/\alpha.
\end{align*}

Then $\{M_t\}$ is a martingale with respect to the natural filtration.
\end{lemma}

%
%
%
\begin{proof}
Let ${\cal F}_{i}$ be the $\sigma$-algebra generated by $M_0,\ldots,M_{i}$.
To show the result, check that 
$\mean[M_t|{\cal F}_{t-1}] = M_{t-1}$.  If kernel ${\bf K}_{1,C}$ is used:
\begin{align*}
\mean[M_t|M_{t-1} = (Cp)^i,{\cal F}_{t-2}] &= p(Cp)^{i-1} + 
  (1-p)\sum_{j=0}^\infty (Cp)^{i + j}(C-1)(1/C)^{j + 1} \\
 &= (Cp)^i\left[C^{-1} + (1 - p)(C - 1)C^{-1}\sum_{j=0}^\infty p^j \right] \\
 &= (Cp)^i[C^{-1} + (C - 1)C^{-1}] \\
 &= (Cp)^i = M_{t-1}.
\end{align*}
When ${\cal F}_{t - 1}$ leads to the use of a ${\bf K}_{2,\alpha}$ kernel:
\begin{align*}
\mean[M_t|M_{t-1} = x,{\cal F}_{t-2}] &= 
(1/\alpha) \alpha x + 0(1 - 1 / \alpha) = x = M_{t-1}.
\end{align*}
\end{proof}

Taking a step in a ${\bf K}_{1,C}$ chain
requires the generation of one $\bern(p)$,
and generation of one geometric random variable with mean
$C/(C - 1)$ if the Bernoulli is 0.  Taking a step in any of the $\alpha$
chains requires no $\bern(p)$ flip, only a single $\bern(1/\alpha)$.

Now everything is in place to solve the original problem of drawing a 
$Cp = (Cp)^1$ coin.  Given $C$ and $\alpha$, fix $\gamma \in (0,1)$ and 
$k > 0$ as
\[
\gamma = 1/2, \ k = \kasfunctionofepsilonandC
\] 
(Any choice of $\gamma \in (0,1)$ gives a correct algorithm, using 
$\gamma = 1/2$ will
give the running time bound in Theorem~\ref{THM:upperbound}.)
Start with $M_0 = Cp$ and take steps using the ${\bf K}_{1,C}$ kernel until
the state is 
either $(Cp)^0$ or $(Cp)^{i}$ where
$i \geq k$.  

At this point, let $\alpha = (1 + \gamma \epsilon)^{i}$, and take one step
in ${\bf K}_{2,\alpha}$.  The result is either the state moves to 0 (in which
case return 0 and quit) or the state moves to 
$(C(1 + \gamma \epsilon) p)^{i}$.  Now start taking moves in
${\bf K}_{1,C(1 + \gamma \epsilon )}$ until either $(C(1 + \gamma)p)^0$ or
$(C(1 + \gamma \epsilon)p)^{i}$ with $i \geq k/(1 - \gamma)$ is reached.

Keep repeating this process, taking away a fraction $\gamma$ of 
$\epsilon$ each time 
and multiplying
$k$ by $(1 - \gamma)^{-1}$ to compensate 
at each step until either a 0 coin state is reached, or the 
${\bf K}_\alpha$ chain returns 0, at which point terminate.

At each step in the algorithm the parameter of the coin being flipped
changes according to one of the Markov chains from Lemma~\ref{LEM:markov},
and so forms a martingale.
In light of Theorem~\ref{THM:martingale}, to prove the 
algorithm is correct requires only that $\prob(T < \infty) = 1$, where
$T$ is the number of calls to $\bern(p)$.  In fact, in the next section
a much stronger result, a bound on $\mean[T]$, is shown.

Note that the ${\bf K}_{1,C}$ and ${\bf K}_{2,\alpha}$ kernels given here
are not the only transitions possible.  There are an uncountable number
of ways to build such a martingale.  As a simple example, for 
$C > 2$ you could flip two $p$-coins initially instead of just one
and then work from there.
The choices made here are for the simplicity of the algorithm and the 
ability to bound the expected running time as shown in the next section.  

\section{Upper bound on running time}
\label{SEC:upper}

Line 4 is the only line in {\sc Linear\_Bernoulli\_Factory}
where a call to $\bern(p)$ is made.  This line is inside a repeat
loop from lines 3--6.  To bound the average number of times this
repeat loop is run, 
let $i_t$ denote the value of $i$ at line 5 after this
line has been executed $t$ times (let $i_0 = 1$).  From
Lemma~\ref{LEM:markov}, $(Cp)^{i_t}$ is a martingale.  
Let $\tau = \inf\{t:i_t = 0 \text{ or } i_t \geq k\}$.  Since $i_{t \wedge \tau}$
is a Markov chain with transient states $\{1,2,\ldots,\lceil k \rceil - 1\}$,
$\prob(\tau < \infty) = 1$.  So $i_\tau$ is either $0$ or at least $k$.

\begin{lemma}
\label{LEM:gamblersruin}
Let $p_k = \prob(i_\tau \geq k)$.  Then
\[
p_k \leq \frac{1 - Cp}{1 - (Cp)^k}.
\]
\end{lemma}

\begin{proof}
Since $(Cp)^{i_t}$ is a martingale, so is $(Cp)^{i_{t \wedge \tau}}$.  So
$\mean[(Cp)^{i_{t \wedge \tau}}] = Cp$, and
\begin{align*}
Cp &= 1 \cdot \prob(i_{t \wedge \tau} = 0) + 
 \prob(i_{t \wedge \tau} \geq k)\mean[(Cp)^i_{t \wedge \tau}|i_{t \wedge \tau} \geq k]
 + \prob(t < \tau)\mean[(Cp)^{i_t}|t < \tau] \\
 &\leq \prob(i_{t \wedge \tau} = 0) + (Cp)^k \prob(i_{t \wedge \tau} \geq k)
       + \prob(t < \tau)(Cp).
\end{align*}
Since this inequality holds for all $t$ and involves probabilities
(which are just expectations of bounded random variables), 
taking the limit as $t$ goes to infinity gives 
\[
Cp \leq [1 - \prob(i_{\tau} \geq k)] + (Cp)^k \prob(i_{\tau} \geq k),
\]
and solving for $p_k = \prob(i_\tau \geq k)$ gives the result.
\end{proof}

\begin{lemma}  
\label{LEM:mean}
For $\{i_t\}$ and $\tau$ as before,
\[
\mean[\tau] \leq \frac{k(C - 1) + C}{1 - (Cp)^k} - \frac{C - 1}{1 - Cp}.
\] 
\end{lemma}

\begin{proof}
When $0 < i_t < k$, 
\[
\mean[i_t|i_{t-1},\ldots,i_0] = i_{t-1} - 1 + (1 - p)\frac{C}{C - 1} = 
  i_{t-1} + \frac{1 - Cp}{C - 1}.
\]
Hence 
\[
i'_t = i_{t \wedge \tau} - [(1 - Cp)/(C - 1)](t \wedge \tau)
\]
is a martingale.  Since $i'_0 = i_{0 \wedge \tau} = 1 = \mean[i'_t]$, 
this gives
\begin{align*}
1 &= \mean[i'_{t}] = \mean[i'_{\tau}\cdot \ind(\tau \leq t)]
 + \mean[i'_{t} \cdot \ind(\tau > t)].
\end{align*}
Each of these terms can be bounded:
\begin{align*}
\mean[i'_{\tau}\cdot \ind(\tau \leq t)]
 &= \mean[(i_{\tau} - [(1 - Cp)/(C - 1)] \tau) \cdot \ind(\tau \leq t)] \\
\mean[i'_t \ind(\tau > t)] &\leq \mean[i_t \ind(\tau > t)].
\end{align*}
Using $\mean[i_\tau \ind(\tau \leq t)] \leq \mean[i_\tau]$ gives
\begin{equation}
\label{EQN:bounds}
1 \leq \mean[i_\tau] 
 - [(1 - Cp)/(C - 1)]\mean[\tau\ind(\tau \leq t)] 
 + \mean[i_{t}\ind(\tau > t)].
\end{equation}

Since $i_{t \wedge \tau}$ is a Markov chain where the
states between 0 and $k$ are transient, $\mean[\tau]$ is finite.
Hence $\mean[\tau \ind(\tau \leq t)] \leq \mean[\tau] < \infty$ for all $t$.

On the last step before $t = \tau$, $i_t$ either moves to 0, or increases by
a geometric random variable with mean $C/(C - 1)$, bringing the bound to 
$(k + C/(C - 1))\prob(i_\tau \geq k)$.  Hence
\begin{align*}
\mean[i_\tau] &=
 \mean[i_\tau |i_\tau \geq k] \prob(i_\tau \geq k) + 0 \cdot \prob(i_\tau = 0) \\ 
&\leq [k + C/(C - 1)](1 - Cp)/(1 - (Cp)^k),
\end{align*}  
using $\prob(i_\tau \geq k) \leq (1 - Cp)/(1 - (Cp)^k)$ from the previous lemma.

Now consider $\mean[i_t \ind(t< \tau)]$.  This is bounded as well because
$i_t \ind(t < \tau)$ must lie in $(0,k)$.  Therefore
the Lebesgue dominated convergence theorem allows taking
the limit as $t \rightarrow \infty$ inside the expectations, giving
\[
1 \leq [k + C/(C - 1)](1 - Cp)/(1 - (Cp)^k) - [(1 - Cp)/(C - 1)]\mean[\tau].
\]
So $\mean[\tau] \leq [k(C - 1) + C]/(1 - (Cp)^k) - (C - 1)/(1 - Cp).$
\end{proof}

The previous lemma bounds the expected number of flips on the first pass through
line 3--6, but this is embedded in a larger repeat loop in lines 2--10.
Let $T_i$ denote the number
of flips on pass $i$ through the larger repeat loop (so $T_1 = \tau$) if
this pass actually occurs.  
Consider the second pass where $i_{2,t}$ is the value of $i$ after $t$ times
through lines 3--6 in the second pass through 2--10.  (Set 
$i_{2,0} = i_{\tau}$ to be the 
ending value of first pass.)  For the second pass to run,
$i_{2,0} \geq k$, but now the parameter value is $C(1 + \gamma \epsilon)p$,
and the $i_{2,t}$ process stops when it is either $0$ or at least 
$k / (1 - \gamma)$.
Let $\tau_2 = \inf\{t:i_{2,t} = 0 \text{ or } i_{2,t} \geq k/(1 - \gamma)\}$.    

\begin{lemma}
\label{LEM:t2}
Let $k_2 = k/(1 - \gamma)$ and 
$C_2 = C(1 + \gamma\epsilon)$.  Then
\[
\mean[T_2] \leq [\gamma k_2(C_2 - 1) + C_2]
   (1 - C_2 p)^{-1}.
\]
\end{lemma}

\begin{proof}  
Since $i_{2,t \wedge \tau_2}$ is a Markov chain where 
all states between 0 and $k_2 = k/(1 - \gamma)$ are transient, $\tau_2$ is
finite and has bounded expectation.
As in Lemma~\ref{LEM:mean}, 
$i'_t = i_{2,t \wedge \tau_2} - [(1 - C_2 p)/(C_2 - 1)]
   (t \wedge \tau)$ is a martingale,
so 
\[
\mean[i_{2,t \wedge \tau_2}] - [(1-C_2 p)/(C_2 - 1)]
  \mean(t \wedge \tau) = i_{2,0} \geq k.
\]
Any step that puts $i_{2,t}$ past $k/(1 - \gamma)$ has expected value 
$C_2/(C_2 - 1)$ so 
$\mean[i_{2,t \wedge \tau_2}] \leq k(1 - \gamma)^{-1} + C_2/(C_2 - 1).$
Hence
\[
\mean(t \wedge \tau) \leq 
  [k[(1 - \gamma)^{-1} - 1] + C_2/(C_2-1)](C_2 - 1)/(1 - C_2 p).
\]
This holds for all $t$, and since $t \wedge \tau$ is an increasing random
variable that converges to $\tau$ with probability 1, the monotone convergence
theorem gives that it converges to $\tau$ as well.  Using
$k[(1 - \gamma)^{-1} - 1] = \gamma k(1 - \gamma)^{-1} = \gamma k_2$ 
finishes the proof.
\end{proof}

Now generalize.  Let $T_j$ be the number of flips in the $j$th stage, where
$i_{j,t}$ reaches 0 or $k/(1 - \gamma)^{j-1}$ in $\tau_j$ time.  Then

\begin{lemma}
Let $k_j = k/(1 - \gamma)^{j-1}$.  Set $C_1 = C$ and for $j \geq 2$ make
$C_j = C_{j-1}(1 + \gamma (1 - \gamma)^{j-2} \epsilon)^{-1}$.  Then
\[
\mean[T_j] = [\gamma k_j
   (C_j - 1) + C_j]
   (1 - C_j p)^{-1}.
\]
\end{lemma}

\begin{proof}
Use induction.  The previous lemma is the base case $j = 2$, and the induction
step is the same as in the proof of the previous lemma.
\end{proof}

Fortunately, the algorithm is unlikely to reach stage $j$ for large $j$.

\begin{lemma}
\label{LEM:stagereached}
The chance that stage $j$ is reached is at most
\[
\exp(-(j-1)\gamma \epsilon k)(1 - Cp)(1 - (Cp)^k)^{-1}.
\]
\end{lemma}

\begin{proof}
From Lemma~\ref{LEM:gamblersruin} there is at least at most a  
$(1 - Cp)/(1 - (Cp)^k)$ chance that $i \geq k$ at the end of the first stage.
Then in order for the second stage to run, $R$ must equal 1 which happens with
probability
\[
(1 + \gamma \epsilon)^{-i} \leq \exp(-i\gamma \epsilon)
 \leq \exp(-\gamma \epsilon k).
\]
At the next stage $\epsilon$ has been multiplied by a factor of $1 - \gamma$,
but $k$ has been multiplied by a factor of $(1 - \gamma)^{-1}$, so 
$k_j \epsilon_j$ 
is constant at all stages.  Hence at stage $\ell \geq 3$, 
the chance of making it to stage $\ell + 1$ is at most 
$\exp(-\gamma \epsilon k)$.
There are $j - 1$ stages to pass through to stage $j$,
so the chance of making it to stage $j$ is at most
$\exp(-(j-1)\gamma \epsilon k)(1 - Cp)(1 - (Cp)^k)^{-1}.$
\end{proof}

\begin{theorem}
\label{THM:preciseupperbound}
For $k = \kasfunctionofepsilonandC$,
the total expected number of coin flips in the algorithm 
is bounded above by
\[
\mean[T] \leq \frac{k(C - 1) + C}{1 - (Cp)^k} - \frac{C - 1}{1 - Cp} 
 + \frac{
      r [\gamma k (C(1-\epsilon)^{-1} - 1) + 
       (1 - \gamma)^2 C(1 - \epsilon)^{-1})]
       (1 - r)^{-1}} {1 - (Cp)^k},
\]
where $r = \exp(-k\epsilon\gamma)(1 - \gamma)^{-2}.$
\end{theorem}

Before proving this theorem, 
it will be helpful to know how $1 - C_jp$ behaves.
In essence, each time the algorithm advances to the next stage,
$1 - C_jp$ is at least $(1 - C_{j-1}p)(1 - \gamma)$.
\begin{lemma}
\[
1 - C_{j}p \geq (1 - Cp) (1 - \gamma)^{j-1}.
\]
\end{lemma}

\begin{proof}  Proceed by induction.  In the base case $j = 1$ 
both sides are equal 
since $C_1 = C$.  Suppose the lemma holds for $C_j$ and consider $C_{j+1}$.

For $x \in (0,1)$ and $\alpha > 1$,
\[
1 - \alpha x \geq (1 - x)(1 - \gamma) \Leftrightarrow \alpha 
  \leq \frac{1 - (1 - x)(1 - \gamma)}{x}.
\]
Set $x = C_j p$ and 
$\alpha = C_{j+1}p/(C_j p) = (1 + \gamma (1 - \gamma)^{j-1}\epsilon)$.
The induction hypothesis upper bounds the 
numerator and lower bounds the denominator:
\[
\frac{1 - (1 - x)(1 - \gamma)}{x} = \frac{1 - (1 - C_j p)(1 - \gamma)}{x} 
 \leq \frac{1 - (1 - C p)(1 - \gamma)^j}
 {1 - (1 - C p)(1 -\gamma)^{j-1}}
\]
Let $y = (1 - Cp)(1 - \gamma)^{j-1}$.  Since
\[
\frac{1 - (1 - \gamma)y}{1 - y} = 1 + \frac{\gamma y}{1 - y},
\]
the question reduces to showing
$\alpha = 1 + \gamma(1 - \gamma)^{j-1}\epsilon \leq 1 + \gamma y/(1 - y)$.
Since $y/(1 - y)$ is an increasing function, and 
$y \geq \epsilon (1 - \gamma)^{j-1}$, this will be true if and only if
\[
(1 - \gamma)^{j-1}\epsilon \leq 
  \frac{(1 - \gamma)^{j-1} \epsilon}{1 - (1 - \gamma)^{j-1}\epsilon}
\]
which is true for $\gamma \in (0,1)$.
\end{proof}

Now a stronger statement can be shown.

\begin{lemma}
\[
1 - C_{j+1}p \geq (1 - C_jp)(1 - \gamma). 
\]
\end{lemma}

\begin{proof}
Using $1 - \alpha x \geq (1 - x)(1 - \gamma) \Leftrightarrow \alpha 
  \leq (1 - (1 - x)(1 - \gamma))x^{-1}.$ for $x \in (0,1)$ and $\alpha > 1$
as in the last lemma, let $x = C_j p$ and 
$\alpha = C_{j+1}p/(C_j p) = (1 + \gamma (1 - \gamma)^{j-1}\epsilon)$ as 
before.  For $y = 1 - C_jp$, the goal is to show
\[
\alpha \leq 
  \frac{1 - (1 - x)(1 - \gamma)}{x} \leq \frac{1 - (1 - C_jp)(1 - \gamma)}
 {1 - (1 - C_jp)} \leq 1 + \frac{\gamma y}{1 - y}
\]
which is true if and only if 
$1 + \gamma(1 - \gamma)^{j-1}\epsilon \leq 1 + \gamma y/(1 - y)$.
As $y/(1 - y)$ is increasing, this is true if and only if it is true
for the smallest value of $y$.  By the previous lemma 
$y = (1 - C_jp) \geq (1 - Cp)(1 - \gamma)^{j-1} \geq \epsilon(1 - \gamma)^{j-1}$.
This gives the result.
\end{proof}

\begin{proof}[Proof of Theorem~\ref{THM:preciseupperbound}]
By the monotone convergence theorem the total expected number of flips is
just the expected number of flips executed at each stage.  By 
Lemma~\ref{LEM:stagereached} the chance of reaching stage $j$
is $\exp(-(j-1)\gamma \epsilon k)(1 - Cp)(1 - (Cp)^k)^{-1}.$
Hence
\[
\mean[T] \leq \mean[\tau] + \frac{1 - Cp}{1 - (Cp)^k} \sum_{j=2}^\infty
  \exp(-(j-1)\gamma\epsilon k) \mean[T_j].
\]

Since $1 - C_jp \leq 1$, $C_j \leq C_\epsilon \defeq C/(1 - \epsilon)$ 
for all $j$.  So for all $j$, 
\[
\mean[T_j] \leq w_j \defeq 
    [\gamma (1 - \gamma)^{-j + 1} k (C_\epsilon - 1) + 
    C_\epsilon](1 - C_j p)^{-1}.
\]
By the previous lemma:
\[
\frac{w_{j+1}}{w_j} = \frac{\gamma (1 - \gamma)^{-j} k 
    (C_\epsilon - 1) + C_\epsilon} 
    {\gamma (1 - \gamma)^{-j+1} k 
    (C_\epsilon - 1) + C_\epsilon} \cdot
    \frac{1 - C_j p}{1 - C_{j+1}p}  \leq (1 - \gamma)^{-2}.
\]
Therefore
\[
\mean[T] \leq \mean[\tau] + \frac{1 - Cp}{1 - (Cp)^k} \cdot \frac{w_2
    \exp(-\gamma\epsilon k)}
   {1 - \exp(-\gamma\epsilon k)(1 - \gamma)^{-2}},
\]
as long as $r \defeq \exp(-\gamma\epsilon k)(1 - \gamma)^{-2} < 1$.

Using $(1 - Cp)(1 - C(1 + \gamma\epsilon)p)^{-1} \leq (1 - \gamma)^{-1}$ gives
\[
(1 - Cp)w_2 \leq (1 - \gamma)^{-1}[
 \gamma(1 - \gamma)^{-1} k (C_\epsilon  -1) + C_\epsilon]
= (1 - \gamma)^{-2}[\gamma k (C_\epsilon - 1) + (1 - \gamma)^2 C_\epsilon].
\]
Hence
\[
\mean[T] \leq \mean[\tau] + (1-(Cp)^k)^{-1}
  \frac{r}{1 - r} [\gamma k  (C_\epsilon - 1) + 
    (1 - \gamma)^2 C_\epsilon ].
\]
Lemma~\ref{LEM:mean} then gives the result. 
\end{proof}

\begin{proof}[Proof of Theorem~\ref{THM:upperbound}]
Theorem~\ref{THM:preciseupperbound} 
holds for any $\gamma$ and $\epsilon$ in $(0,1)$, 
as long as $k$ makes $r < 1$.  This, however, is very 
difficult to optimize for general $C$ and $\epsilon$.  Since $C = 2$
is so important to the general analytic case, consider optimal values
for this case.
Since $\exp(-\gamma k \epsilon)$ must be small, $k$ will always contain
an $\epsilon^{-1}$ factor.  Writing $k$ as 
$k = (\gamma \epsilon)^{-1} m$ gives $\exp(-\gamma k \epsilon) = \exp(-m)$.

Optimizing the bound over $m$ and $\gamma$ gives (when $\epsilon = 0.2$)
$m \approx 2.3$ and $\gamma \approx 1/2$.  So
\[
k = 4.6 \epsilon^{-1},\ r = \exp(-2.3)(1-1/2)^{-2} = 0.4010354\ldots.
\]

Note
$(Cp)^k \leq (1 - \epsilon)^{4.6/\epsilon} \leq \exp(-4.6)$, so
$1 - Cp \leq 1 - (Cp)^k \geq 0.99$.

Putting this in the bound from Theorem~\ref{THM:preciseupperbound} gives
\[
\mean[T] \leq 
 \frac{4.6 \epsilon^{-1}(C - 1) + 1}{0.99} + 
 \frac{0.4011/0.99}{1 - 0.4011} 
  \left[2.3\epsilon^{-1}\left(\frac{C}{1 - \epsilon} - 1\right)
   + \frac{C/4}{1 - \epsilon} \right].
\]
When $\epsilon \geq 0.644$, the algorithm can be run with $\epsilon = 0.644$,
which makes $(1 - \min\{0.644,\epsilon\})^{-1} \leq 2.809$ and the overall
bound at most $9.5C\epsilon^{-1}$.

%

\end{proof}

Of course, given Theorem~\ref{THM:preciseupperbound}, a better way to run
line 1 in the algorithm is given $C$ and $\epsilon$, 
choose $m$ and $\gamma$ (and then $k = m(\gamma \epsilon)^{-1}$) 
to minimize the expected running time bound.

In~\cite{thomasb2011}, computer experiments were run to determine the 
effectiveness of several Bernoulli factories.  Their method, the 
Thomas-Blanchet Cascade approach, proved the fastest in computer 
experiments.
  
Figure~\ref{FIG:data} 
includes this data from~\cite{thomasb2011} and adds three columns:
the optimal $(m,\gamma)$ values given $C$ and $\epsilon$, 
a theoretical bound for the expected flips using 
Theorem~\ref{THM:preciseupperbound}, and an experimental bound found 
through computer experiments simulating the coin 10000 times.  The
value of $\epsilon$ was $0.2$ throughout.  The results are reported as
(mean, standard deviation).  (Note $m^*$ and $\gamma^*$ are the
near optimal $m$ and $\gamma$ values for a particular value of $C$.)
The results show not only a lower mean for the algorithm of 
Section~\ref{SEC:algorithm}, but a much reduced 
standard deviation as well.

\begin{figure}
\begin{tabular}{cccccccccc}
\toprule
$C$ & $(m^*,\gamma^*)$ & 
  \parbox{1.5in}{\centering Theory bound \\ (new method)} & 
  \parbox{1in}{\centering Exper. \\ (new method)} & 
  \parbox{1in}{\centering Thomas- \\ Blanchet}  \\ 
\midrule
2  & $(2.31,0.463)$ & $35.56$ & $(28,43)$ & $(66,512)$ \\
5  & $(2.01,0.425)$ & $133.7$ & $(107,62)$ & $(246,1215)$ \\
10 & $(1.91,0.410)$ & $296.9$ & $(239,140)$ & $(614,1851)$ \\
20 & $(1.81,0.394)$ & $623.2$ & $(516,426)$ & $(1410,3047)$ \\
\end{tabular}
\caption{Comparison of number of flips used by Bernoulli factories.
The ``Theory bound'' column comes from Theorem~\ref{THM:preciseupperbound}.
The ``Exper.'' column gives the experimental
expected value found by running the algorithm 10000 times, and the 
Thomas-Blanchet results come from~\cite{thomasb2011}.  The 
simulation results are
given as (sample mean, sample standard deviation).}
\label{FIG:data}
\end{figure}

In addition, the results show that the optimal $k$ and $\gamma$ values are
relatively insensitive to the value of $C$, and that the precise bound of
Theorem~\ref{THM:preciseupperbound} is relatively close to the true behavior
of the algorithm.

\section{Lower bound on running time}
\label{SEC:lower}

In order to establish the lower bound, it is helpful to consider the 
more general problem of developing a randomized algorithm for 
estimating the value of $p$.

\begin{definition}
Given $p^* \in (0,1]$ an {\em estimation algorithm} 
${\cal B}$ is a computable function that
takes input $U \sim \unif([0,1])$ together with 
$X_1,X_2,\ldots \sim \bern(p)$ iid and returns $\hat p \in [0,1]$.
Let $T_{\cal B}$ be the infimum of times $t$ such that the value of 
$\hat p$ only depends on $X_1,\ldots,X_t$. 
Call $T_{\cal B}$ the {\em running time}
of the algorithm.
\end{definition}

The goal is to make $\hat p$ an accurate estimate of $p$.  The
definition below comes from~\cite{dagumklr2000}.

\begin{definition}
Call $\cal B$ with output $\hat p$ 
an {\em $(\alpha,\beta)$-approximation 
algorithm} if 
\[
\prob(p(1 - \alpha) \leq \hat p \leq p(1 + \alpha)) \geq 1 - \beta.
\]
\end{definition}

Dagum, Karp, Luby and Ross~\cite{dagumklr2000} lower bounded the 
expected running time of any such algorithm ${\cal B}$ for fixed
$\epsilon$.  Their result is very general, and applies for any
distribution in $[0,1]$.  That means it applies to the $\{0,1\}$ coin flips
considered here.
The following form of their lemma takes their general
result and applies it to the special case of $\{0,1\}$ random variables.

\begin{lemma}[Lemma 7.5 of~\cite{dagumklr2000} applied to $X_i \sim \bern(p)$]
Any $(\alpha,\beta)$ approximation algorithm that uses $T_B$ flips must 
satisfy
\[
\mean[T_{\cal B}] \geq (1 - \beta)(1 - \alpha)^2 
  \ln\left(\frac{2 - \beta}{\beta}\right)
 \frac{1 - p}{4e^2 \alpha^2 p}.
\]
\end{lemma}

To apply this lemma, 
fix $C \geq 1$ and $\epsilon > 0$, and consider the following
approximation algorithm for $p$.
\begin{center}
\begin{tabular}{rl}
\toprule
\multicolumn{2}{l}{{\sc Bernoulli\_Factory\_Approximation}\ \ 
 {\em Input:} $C,$ $\bern(p)$.} \\
\midrule
1) & $A \leftarrow 0$, $S \leftarrow 0$ \\
2) & Repeat \\
3) & \hspace*{1em} Draw $W \leftarrow \bern(Cp)$ \\
4) & \hspace*{1em} $A \leftarrow A + 1$, $S \leftarrow S + W$. \\
5) & Until $S = 4$. \\
6) & $\hat p \leftarrow 4/(CA)$ \\
\bottomrule
\end{tabular}
\end{center}

\begin{lemma}
The above algorithm is an $(\epsilon^{1/2},1/4)$-approximation algorithm
for $p \in [0,1/2]$, and the
expected number of $\bern(p)$ draws is $4 \mean[T] / (Cp)$.
\end{lemma}

\begin{proof}
Here $A$ has a negative binomial distribution with parameters $4$ and 
$Cp$ and support $\{4,5,\ldots\}$.  The mean and standard deviation 
of $CA/4$ are
\[
\mean[CA/4] = 1/p, \ \ \sd(CA/4) = (C/4)((1 - Cp))^{1/2}/(Cp)
   \leq \epsilon^{1/2}/(4p).
\]
Hence by Chebyshev's inequality
\begin{align*}
\prob(1/p - 2 \epsilon^{1/2}/(4p) 
      \leq CA/4 \leq 1/p + 2 \epsilon^{1/2}/(4p)) &\geq 1 - 1/2^2 \\
\prob(1 - \epsilon^{1/2}/2
      \leq p CA/4 \leq 1 + \epsilon^{1/2}/2) &\geq 1 - 1/4 \\
\prob((1 - \epsilon^{1/2}/2)^{-1} \geq 4/(p C A) \geq 
      (1 + \epsilon^{1/2}/2)^{-1} 
    &\geq 1 - 1/4 \\
\prob(p(1 + \epsilon^{1/2}) \geq 4/(C A) \geq p(1 - \epsilon^{1/2})) 
    &\geq 1 - 1/4,
\end{align*}
where the last line used the fact that $1/(1 - x) \leq 1 + 2x$ and
$1/(1 + x) \geq 1 - 2x$ for all $x \in [0,1/2]$.

Finally, the expected number of $\bern(p)$ 
draws is just $4/(Cp)$ times the expected
number of draws of $\bern(p)$ to generate one $\bern(Cp)$ coin.
\end{proof}

\begin{proof}[Proof of Theorem~\ref{THM:lowerbound}]
From the previous two lemmas,
\[
\frac{4 \mean[T]}{Cp} \geq (1 - 1/4)(1 - \epsilon^{1/2})^2
 \ln\left(\frac{2 - 1/4}{1/4}\right)
 \frac{1 - p}{4e^2 \epsilon p}.
\]

If $p \in [0,(1 - \epsilon )/C]$ then
$1 - p \geq (1 - (1 - \epsilon)C^{-1})$,
and solving for $\mean[T]$ gives the lower bound.
\end{proof}

\section{Conclusions}

This algorithm is the first Bernoulli factory for $f(p) = Cp$ and
$Cp \leq 1 - \epsilon$ that
has a provable upper bound for the expected number of flips.  Moreover,
this algorithm is both simple to implement and faster in practice than
existing methods.  Furthermore, when considering the running time for
all such factories that work for all 
$p \in [0,(1 - \epsilon)/C]$, this method comes with a constant factor
of the best possible expected running time.

%

\bibliographystyle{plain}

\begin{thebibliography}{1}

\bibitem{asmussengt1992}
S.~Asmussen, P.~W. Glynn, and H.~Thorisson.
\newblock Stationarity detection in the initial transient problem.
\newblock {\em ACM Trans. Modeling and Computer Simulation}, 2(2):130--157,
  1992.

\bibitem{dagumklr2000}
P.~Dagum, R.~Karp, M.~Luby, and S.~Ross.
\newblock An optimal algorithm for {M}onte {C}arlo estimation.
\newblock {\em Siam. J. Comput.}, 29(5):1484--1496, 2000.

\bibitem{durrett2005}
R.~Durrett.
\newblock {\em Probability: Theory and Examples, Third edition}.
\newblock Brooks/Cole, 2005.

\bibitem{flegalh2012}
J.~Flegal and R.~Herbei.
\newblock Exact sampling for intractable probability distributions via a
  {B}ernoulli factory.
\newblock {\em Electron. J. Stat.}, 6:10--37, 2012.

\bibitem{keaneo1994}
M.~S. Keane and G.~L. O'Brien.
\newblock A {B}ernoulli factory.
\newblock {\em ACM Trans. Modeling and Computer Simulation}, 4:213--219, 1994.

\bibitem{latuszynskikpr2011}
K.~{\L}atuszy\'nski, I. Kosmidis, O. Papaspiliopoulos, and G. O. Roberts.
\newblock Simulation events of unknown probability via reverse time
  Martingales.
\newblock {\em Random Structures Algorithms}, 38:441--452, 2011.

\bibitem{nacup2005}
S.~Nacu and Y.~Peres.
\newblock Fast simulation of new coins from old.
\newblock {\em Ann. Appl. Probab.}, 15(1A):93--115, 2005.

\bibitem{thomasb2011}
A.~C. Thomas and J.~Blanchet.
\newblock A practical implementation of the {B}ernoulli factory.
\newblock 2011.
\newblock arXiv:1105.2508.

\bibitem{vonneumann1951}
J.~von Neumann.
\newblock Various techniques used in connection with random digits.
\newblock In {\em Monte Carlo Method}, Applied Mathematics Series 12. National
  Bureau of Standards, 1951.

\end{thebibliography}

\end{document}